\documentclass{article}
\usepackage{amsmath}
\usepackage{amssymb}
\usepackage{amsthm}
\usepackage{bbm}
\usepackage{cite}
\usepackage{leftidx}
\usepackage{pgf}
\usepackage{color}
\usepackage{wasysym}

\usepackage{tikz-cd}
\usepackage{verbatim}
\usepackage{xypic}
\usepackage{upgreek}

\usepackage{hyperref}

\author{Saul Glasman}

\title{Goodwillie calculus and Mackey functors}

\newcommand{\angs}[1]{\langle #1 \rangle}
\newcommand{\Alg}{\mathbf{Alg}}

\newcommand{\bb}{\mathbb}

\newcommand{\Cat}{\textbf{Cat}}
\newcommand{\CMon}{\textbf{CMon}}

\newcommand{\colim}[1]{\underset{#1}{\text{colim }}}

\newcommand{\D}{\Delta}

\newcommand{\eps}{\epsilon}
\newcommand{\End}{\mb{End}}
\newcommand{\F}{\mc{F}}

\newcommand{\Fun}{\text{Fun}}
\newcommand{\Ga}{\Gamma}
\newcommand{\ho}{\text{ho}}

\newcommand{\id}{\text{id}}
\newcommand{\im}{\text{im }}
\newcommand{\inc}{\subseteq}
\newcommand{\inj}{\hookrightarrow}

\newcommand{\iy}{\infty}

\newcommand{\La}{\Lambda}

\newcommand{\mb}{\mathbf}
\newcommand{\mc}{\mathcal}

\newcommand{\Map}{\text{Map}}

\newcommand{\Om}{\Omega}
\newcommand{\op}{\text{op}}
\newcommand{\os}{\overset}
\newcommand{\ot}{\otimes}
\newcommand{\Psh}{\mb{Psh}}

\newcommand{\Si}{\Sigma}

\newcommand{\Sp}{\textbf{Sp}}

\newcommand{\td}{\widetilde}
\newcommand{\toe}{\overset{\sim} \to}
\newcommand{\twa}{\widetilde{\mathcal{O}}}
\newcommand{\Top}{\textbf{Top}}
\newcommand{\ul}{\underline}

\newcommand{\X}{\times}

\newcommand{\Add}{\bb{A}dd}
\newcommand{\Comm}{\text{Comm}}
\newcommand{\cart}{\text{cart}}

\newcommand{\surj}{\text{surj}}
\newcommand{\Fs}{\F_\surj}
\newcommand{\Mack}{\mb{Mack}}

\theoremstyle{definition}

\newtheorem{cor}[subsection]{Corollary}
\newtheorem{dfn}[subsection]{Definition}
\newtheorem{exa}[subsection]{Example}

\newtheorem{lem}[subsection]{Lemma}
\newtheorem{prop}[subsection]{Proposition}
\newtheorem{rec}[subsection]{Recollection}
\newtheorem{rem}[subsection]{Remark}

\newtheorem{thm}[subsection]{Theorem}

\begin{document}
\maketitle
\begin{abstract}
	We show that the category of $n$-excisive functors from the $\iy$-category of spectra to a target stable $\iy$-category $\mb{E}$ is equivalent to the category of $\mb{E}$-valued Mackey functors on an indexing category built from finite sets and surjections. This new classification of polynomial functors arises from an investigation of the structure present on cross effects. The path to this result involves a pair of surprising extension theorems for polynomial functors and a discussion of some interesting topics in semiadditive $\iy$-category theory, including a formula for the free semiadditive $\iy$-category on an $\iy$-category. Our equivalence forms the basis for a set of strong analogies between functor calculus and equivariant stable homotopy theory.
\end{abstract}
\section{Introduction}

This paper is devoted to proving that $n$-excisive functors, in the sense of Goodwillie \cite{Goo91}, from the category $\Sp^{\omega}$ of finite spectra to a target stable $\iy$-category $\mb{E}$ can be represented as $\mb{E}$-valued Mackey functors on a particular indexing category, much like genuine $G$-spectra in $\mb{E}$ for a finite group $G$ \cite{GM, Bar14}. This provides an entirely new classification of $n$-excisive functors in the stable setting, and when coupled with the results of \cite{Gla15a} it recontextualizes the classification given by Arone and Ching \cite{AC15}. It is one of the key examples in both the theory of epiorbital categories developed in \cite{Gla15a} and the wider theory of orbital categories that serves as the foundation for parametrized higher algebra in \cite{BDGNS1} and its sequels. It is a topological analogue of an algebraic result of Baues, Dreckmann, Franjou and Pirashvili \cite{BDFP}.

The indexing category in question - the category that plays the role of the orbit category of the group $G$ in the analogy with equivariant stable homotopy theory - is the category $\Fs^{\leq n}$ of finite sets of cardinality at most $n$ and surjective maps. Then the category $\Mack(\Fs^{\leq n}, \mb{E})$ is by definition the category of additive functors from the semiadditive $\iy$-category 
\[A^{eff}(\F_{\Fs^{\leq n}}),\] the effective Burnside category of the formal coproduct completion of $\Fs^{\leq n}$, to $\mb{E}$. A rough statement of the main theorem is then
\begin{thm}[\ref{thm:main}]
	There is an explicit equivalence of $\iy$-categories from $\Mack(\Fs^{\leq n}, \mb{E})$ to the category $\Fun^{n-exc}(\Sp^\omega, \mb{E})$ of $n$-excisive functors from the category of finite spectra to $\mb{E}$.
\end{thm}
This equivalence is such that if $F : \Sp^{\omega} \to \mb{E}$ is an $n$-excisive functor and $M_F$ is the corresponding Mackey functor, then for any $U \in \Fs^{\leq n}$, the value
\[ M_F(U) \simeq cr_U F(\bb{S}, \bb{S}, \cdots, \bb{S}), \]
the $U$-indexed cross effect of $F$ evaluated on copies of the sphere spectrum.

This work was partly motivated by the following fascinating observation, known to Arone, Ching and McCarthy. Let $\Sp$ be the $\iy$-category of spectra and let 
\[ F = ((-)^{\wedge C_2})^{C_2} : \Sp \to \Sp \]
be the functor obtained as the genuine $C_2$ fixed points of the Hill-Hopkins-Ravenel norm for the group $C_2$ \cite{HHR}. On the other hand, let 
\[ G = P_2(\Sigma^{\iy} \Omega^{\iy}) : \Sp \to \Sp \]
be the 2-excisive approximation in the sense of Goodwillie \cite{Goo03}. Then the functors $F$ and $G$, although their origins are utterly different, are actually equivalent. This prompts one to ask the following question, posed to us by Hopkins and Lurie via Barwick: what happens for higher $n$? Is there an equivariant description of $P_n(\Sigma^{\iy} \Omega^\iy)$?

To address this, we note that Theorem \ref{thm:main} brings functor calculus under the purview of the theory of parametrized higher algebra \cite{BDGNS}, which allows us to port all of our tricks from equivariant stable homotopy theory to functor calculus. This affords us a `norm' that builds an $n$-excisive functor 
\[\mb{Nm}_n(X) : \Sp \to \mb{E}\]
from an object $X \in \mb{E}$ in much the same way as the Hill-Hopkins-Ravenel norm builds a $G$-spectrum $X^{\wedge G}$ from $X$. 

We have been told by Tomer Schlank that this norm can be used to give an explicit inverse to the functor of \ref{thm:main}: if $M \in \Mack(\Fs^{\leq n}, \Sp)$ and $F_M : \Sp \to \Sp$ is the corresponding $n$-excisive functor, then
\[ F_M(X) \simeq (\mb{Nm}_n(X) \ot M)(\angs{1}), \]
where $\ot$ is the Day convolution of Mackey functors and $\angs{1}$ is a one-element set. 

As a special case, we may make the striking observation that the orbital categories $\mb{O}_{C_2}$ and $\Fs^{\leq 2}$ (minus the empty set, whose contribution is unimportant) are equivalent, giving an equivalence 
\[\Sp^{C_2} \toe \Mack(\Fs^{\leq 2}, \Sp)\] 
which is compatible with norms. In particular,
\[ \mb{Nm}_n(\bb{S}) \simeq P_n(\Si^\iy \Omega^\iy) : \Sp \to \Sp, \]
since the norm is symmetric monoidal and $P_n(\Si^\iy \Omega^\iy)$ is the unit for the Day convolution on $\Fun^{n-exc}(\Sp, \Sp)$. This neatly contextualizes the equivalence between fixed points and the 2-excisive approximation, as well as describing the situation for higher $n$. Moreover, the assignment
\[ A \mapsto F_A, \hspace{20pt} F_A(X) = (X^{\wedge C_2} \wedge A)^{C_2} \]
defines an equivalence  $\Sp^{C_2} \toe \Fun^{2-exc}(\Sp, \Sp)$, a memorable result in itself.

The equivalence of Theorem \ref{thm:main} is written as a long composite of functors, each of which is constructed and proved to be an equivalence more or less independently, and so this paper has a modular structure in which each section is both largely self-contained and, we hope, individually interesting. We now review this structure. In Section \ref{sec:semiadd}, we review the theory of semiadditive $\iy$-categories and give a new formula for the free semiadditive $\iy$-category on an $\iy$-category $\mb{C}$. This category has a particularly nice form when $\mb{C}$ is the effective Burnside category $A^{eff}(\F)$ of finite sets. In Section \ref{sec:adjret}, we learn that certain retractions whose existence we can count on in the additive setting may be absent in a semiadditive $\iy$-category, and we discuss how to adjoin missing retractions universally. In Section \ref{sec:calcburn}, we prove results, some of which are folklore, implying that $n$-excisive functors on the category of finite spectra can be reduced to certain combinatorially defined objects. In Section \ref{sec:anacon}, we prove an intriguing analytic continuation theorem. Finally, in Section \ref{sec:synth}, we string all of our equivalences together and prove the main theorem.

The reader may notice a formal similarity between the analytic continuation theorem, Theorem \ref{thm:anacon}, and Proposition \ref{prop:connsp}, which extends polynomial functors from connective spectra to all spectra. In forthcoming joint work with Clark Barwick, Akhil Mathew and Thomas Nikolaus, we will explore this connection and apply it to the study of polynomial functoriality.

We now pause to note some connections with previous work. There is a body of important work on the classification of $n$-excisive functors, notably including \cite{AC15}, \cite{AC14}, \cite{McCN}, and an unpublished theorem of Dwyer and Rezk which appears as Proposition 3.15 and Theorem 3.82 in \cite{AC14}. The latter theorem should be considered the analog of our main result for functors from $\Top$ to $\Sp$. The former two references have the advantage of dealing with polynomial functors for which the source, target or both is unstable, and pertain primarily to structure on the derivatives rather than the cross effects. This makes them thematically somewhat distinct from our work. However, we note that \cite[Proposition 4.24]{AC14} - which, for $L = 0$, gives the cross effects of a functor $F : \Sp \to \Sp$ the structure of a right module over the nonunital commutative operad \textbf{Com} - combines with \cite[Example 1.17]{AC14} to give a functor
\[ d^0(F): \Fs^\op \to \Sp\]
whose values are the cross effects of $F$. Presumably $d^0(F)$ satisfies
\[ d^0(F) \simeq DR^0(F \circ \Sigma^\iy) : \Fs^\op \to \Sp \]
where $DR^0(F \circ \Sigma^\iy)$ is the functor associated to $F \circ \Sigma^\iy : \Top \to \Sp$ by the Dwyer-Rezk classification. This is the contravariant half of the Mackey functor structure we give.

Moreover, the penultimate paragraph of the introduction of \cite{AC14} in some sense anticipates the present paper: we suspect that our main result can be viewed as an ``unrolling" of \cite[Theorem 7.6]{McCN}, which exhibits the \emph{sum} of the cross effects of an $n$-excisive functor as a module over a certain ring. If so, McCarthy's result could be reconstructed from ours using the Schwede-Shipley theorem \cite[Theorem 7.1.2.1]{HA} together with an analysis of the endomorphism ring of the unit $n$-excisive functor $P_n(\Si^\iy_+ \Om^\iy)$.

Finally, some important points of notation and convention. Polynomial functors with source a presentable category will be assumed to preserve filtered colimits. When the target category $\mb{E}$ is presentable, the category of $n$-excisive functors from $\Sp^{\omega} \to \mb{E}$ is equivalent to the category of functors from the category $\Sp$ of all spectra to $\mb{E}$ which are $n$-excisive and preserve filtered colimits, and we will sometimes use this equivalence implicitly. All ``categories" or ``$\iy$-categories" will be quasicategories unless otherwise specified, and thus all categorical constructions will be of the homotopy invariant sort. $\F$ is the category of finite sets and $\F_*$ is the category of finite pointed sets.

We thank Clark Barwick, Akhil Mathew, Tomer Schlank and the participants of the Bourbon Seminar for many helpful conversations related to the subject matter of this paper.

This paper was partly written while the author was supported by the National Science Foundation under agreement no. DMS-1128155. Any opinions, findings and conclusions, or recommendations expressed in this material, are those of the author and do not necessarily reflect the views of the National Science Foundation.

\section{Semiadditive $\iy$-categories and the free semiadditive $\iy$-category on an $\iy$-category}\label{sec:semiadd}
We'll start this section by collecting, for convenience, a few more-or-less well-known facts about semiadditive $\iy$-categories. Most of these results also appear in \cite[\S 2]{GGN}, although those authors use the term ``preadditive" rather than ``semiadditive". Throughout this paper, $\F_*$ will denote the category of finite pointed sets, and if $S$ is a finite pointed set, $S^\circ$ will denote the set of non-basepoint elements of $S$.
\begin{dfn}[\cite{HA}, Definition 6.1.6.13]
By a \emph{semiadditive $\iy$-category}, we mean an $\iy$-category which admits a zero object, finite products and finite coproducts and in which the natural map
\[\left( \begin{tikzcd}[column sep = tiny, row sep = tiny] \id & 0 \\ 0 & \id \end{tikzcd} \right) : X \amalg Y \to X \X Y\]
is an equivalence for any objects $X, Y$. In this situation, we'll tend to use the notation $X \oplus Y$ for both $X \amalg Y$ and $X \X Y$ and leave their equivalence implicit.
\end{dfn}

The semiadditive axiom is the line that divides combinatorics from algebra. As we'll shortly recall, the mapping spaces in a semiadditive $\iy$-category carry canonical commutative monoid structures.

\begin{lem}\label{lem:cartcocart}
Let $\mb{C}^\ot$ be a symmetric monoidal category. Then the following conditions are equivalent:
\begin{enumerate}
\item $\mb{C}^\ot$ is both cartesian and cocartesian in the sense of \cite[Definition 2.4.0.1]{HA}.
\item $\mb{C}^\ot$ is either cartesian or cocartesian in the sense of \cite[Definition 2.4.0.1]{HA}, and the underlying category $\mb{C}$ is semiadditive.
\end{enumerate}
\end{lem}
The proof is immediate. \qed

If $\mb{C}$ is semiadditive, we'll denote the cartesian, or equivalently the cocartesian, symmetric monoidal category associated to $\mb{C}$ by $\mb{C}^\oplus$.
\begin{cor}
Let $\mb{SMCat}^{\oplus}_\iy$ be the $\iy$-category of symmetric monoidal $\iy$-categories which are both cartesian and cocartesian and symmetric monoidal functors, and let $\Cat^\oplus_\iy$ be the $\iy$-category of semiadditive $\iy$-categories and direct-sum-preserving functors. Then the forgetful functor
\[\theta : \mb{SMCat}^{\oplus}_\iy \to \Cat^\oplus_\iy\]
is an equivalence.
\end{cor}
\begin{proof}
Combine Lemma \ref{lem:cartcocart} with \cite[Corollary 2.4.1.9]{HA}.
\end{proof}
\begin{cor}
Let $\CMon(\mb{C})$ be the category of commutative monoids in $\mb{C}$ \cite[Remark 2.4.2.2]{HA}. If $\mb{C}$ is semiadditive, then the forgetful functor $u : \CMon(\mb{C}) \to \mb{C}$ is a trivial Kan fibration.
\end{cor}
\begin{proof}
We have a homotopy commutative diagram
\[\begin{tikzcd}\Alg_{\F_*}(\mb{C})  \ar{d}{\alpha} \ar{r}{\beta} & \CMon(\mb{C})  \ar{dl}{u} \\ \mb{C} \end{tikzcd}\]
where $\alpha$ is the restriction functor and $\beta$ arises from the cartesian structure on $\mb{C}^\oplus$ \cite[Definition 2.4.1.1]{HA}.
$u$ is always a Kan fibration. Since $\mb{C}^\oplus$ is cocartesian, $\alpha$ is an equivalence \cite[Proposition 2.4.3.16]{HA}, and since $\mb{C}^\oplus$ is cartesian, $\beta$ is an equivalence \cite[Proposition 2.4.2.5]{HA}, so $u$ is an equivalence.
\end{proof}

\begin{lem}\label{lem:semaddenr}
Semiadditive categories are naturally enriched in commutative monoids in the following sense: if $\mb{C}$ is semiadditive, then the functor 
\[\Map(-, -): \mb{C}^\op \X \mb{C} \to \Top\]
 extends canonically over the forgetful functor $\CMon \to \Top$.
\end{lem}
\begin{proof}
Since $\mb{C}^\op \X \mb{C}$ is also semiadditive, we have a diagram
\[\begin{tikzcd}
\CMon(\mb{C}^\op \X \mb{C}) \ar{d}{\sim} \ar{r} & \CMon \ar{d} \\
\mb{C}^\op \X \mb{C} \ar{r} & \Top.
\end{tikzcd}\]
\end{proof}

\begin{lem} \label{lem:cmonsemiadd}
If $\mb{C}$ is any $\iy$-category with finite products, the category $\CMon(\mb{C})$ of commutative monoids in $\mb{C}$ is semiadditive.
\end{lem}

\begin{proof}
The category $\Alg_{\Comm}(\mb{C}^\X)$ of \cite[Example 3.2.4.4]{HA} has underlying $\iy$-category equivalent to $\CMon(\mb{C})$ \cite[Proposition 2.4.3.16]{HA} and is cocartesian \cite[Proposition 3.2.4.7]{HA}. We'll prove that $\Alg_\Comm(\mb{C}^\X)$ is also cartesian. Unwinding the definitions and using the characterization of the cocartesian edges of $\Alg_\Comm(\mb{C}^\X)$ given in \cite[Proposition 3.2.4.3 (4)]{HA}, we find that the unit object in $\CMon(\mb{C})$ is the final object of $\mb{C}$ with its unique commutative monoid structure, and that the tensor product $\CMon(\mb{C}) \X \CMon(\mb{C}) \to \CMon(\mb{C})$ is the pointwise product. The result follows.
\end{proof}
\begin{rem} \label{rem:cmonext}
	By the same reasoning as in the proof of Lemma \ref{lem:semaddenr}, if $\mb{C}$ is semiadditive and $\mb{D}$ admits finite products, then the category $\Fun^\X(\mb{C}, \mb{D})$ of product-preserving functors from $\mb{C}$ to $\mb{D}$ is equivalent to the category of $\Fun^\oplus(\mb{C}, \CMon(\mb{D}))$ of additive functors from $\mb{C}$ to the category of commutative monoids in $\mb{D}$.
\end{rem}
Let $\mb{C}$ be any $\iy$-category. We now turn to the problem of describing the free semiadditive $\iy$-category on $\mb{C}$: can we give a formula for the left adjoint of the forgetful functor $\Cat^\oplus_\iy \to \Cat_\iy$?

We'll start by giving a fairly abstract answer to this question; although we'll end up with something more explicit, the following result is a necessary waypoint.
\begin{prop}
Define
\[\sphericalangle : \mb{C} \to \Fun(\mb{C}^\op, \Top) \to \Fun(\mb{C}^\op, \CMon)\]
as the composition of the Yoneda embedding with the pointwise free commutative monoid functor. By Lemma \ref{lem:cmonsemiadd}, $\Fun(\mb{C}^\op, \CMon)$ is semiadditive; let $\Add^p(\mb{C})$ denote the closure of the essential image of $\sphericalangle$ under direct sums. Then
\[\sphericalangle : \mb{C} \to \Add^p(\mb{C})\]
exhibits $\Add^p(\mb{C})$ as the free semiadditive $\iy$-category on $\mb{C}$.
\end{prop}
\begin{proof}
 Since $ \Fun(\mb{C}^\op, \Top)$ is the free presentable $\iy$-category on $\mb{C}$, \\ $\Fun(\mb{C}^\op, \CMon)$ is the free semiadditive presentable $\iy$-category on $\mb{C}$ by \cite[Theorem 4.6]{GGN}, in the sense that if $\mb{P}$ is a presentable semiadditive $\iy$-category, then the restriction functor
\[\Fun^L(\Fun(\mb{C}^\op, \CMon), \mb{P}) \to \Fun(\mb{C}, \mb{P})\]
is an equivalence, where $\Fun^L$ denotes the category of colimit-preserving functors.

On the other hand, the free semiadditive presentable $\iy$-category on $\mb{C}$ is also the nonabelian derived category (c.f. \ref{rec:nonabder}) of the free semiadditive $\iy$-category on $\mb{C}$, by \cite[Proposition 5.5.8.15]{HTT}, and so it contains the free semiadditive category on $\mb{C}$ as a full subcategory. This full subcategory must contain all the objects of $\Add^p(\mb{C})$, and since $\Add^p(\mb{C})$ is already semiadditive, it coincides with the free semiadditive category on $\mb{C}$.
%

\end{proof}
\begin{lem}\label{lem:Addpmapsp}
For each $X, Y \in \mb{C}$, the map
\[\sphericalangle_{X. Y} : \Map_{\mb{C}}(X, Y) \to \Map_{\Add^p(\mb{C})}(\sphericalangle(X), \sphericalangle(Y))\]
exhibits the target as the free commutative monoid on the source.
\end{lem}
\begin{proof}
Let $h_X$ denote the representable functor 
\[\Map(- ,X) \in \Fun(\mb{C}^\op, \Top).\]
We have
\begin{align*} \Map_{\Add^p(\mb{C})}(\sphericalangle(X), \sphericalangle(Y)) & \simeq \Map_{\Fun(\mb{C}^\op, \Top)}(h_X, \sphericalangle(Y)) \\
& \simeq (\sphericalangle(Y))(X), \end{align*}
which gives the result, since the free commutative monoid functor is pointwise.
\end{proof}
We'll now be able to construct our more explicit model for the free semiadditive category on $\mb{C}$.
\begin{dfn}We denote by $\F_\mb{C}$ the category obtained by adjoining formal finite coproducts to $\mb{C}$. Explicitly, $\F_\mb{C}$ is the full subcategory of $\Fun(\mb{C}^\op, \Top)$ spanned by the finite coproducts of representable presheaves. The notation derives from the intuition that objects of $\F_\mb{C}$ should be thought of as finite sets of objects of $\mb{C}$.
\end{dfn}
We'll give an alternative construction of $\F_\mb{C}$ which is more convenient for our purposes. Let $\F \inj \F_*$ be the usual subcategory inclusion, and define 
\[\mb{C}^\amalg_\F :=\mb{C}^\amalg \X_{\F_*} \F,\]
where $\mb{C}^\amalg$ is the $\iy$-operad of \cite[Construction 2.4.3.1]{HA}. An object of $\mb{C}^\amalg_\F$ is a finite set $S$ together with an $S$-tuple $(X_s)_{s \in S}$ of objects of $\mb{C}$, and a morphism 
\[f : (S, (X_s)_{s \in S}) \to (T, (Y_t)_{t \in T}) \]
is a set morphism $f_S : S \to T$ together with a morphism $f_s: X_s \to Y_{f_S(s)}$ for each $s \in S$. 

We'll denote the projection $\mb{C}^\amalg_\F \to \F$ by $\pi_0^\mb{C}$; since the projection $\Ga^* \X_{\F_*} \F \to \F$ \cite[Construction 2.4.3.1]{HA} is a cocartesian fibration, it follows from \cite[Corollary 3.2.2.13]{HTT} that $\pi_0^\mb{C}$ is a cartesian fibration, and that a morphism $f$ is cartesian if and only if $f_s$ is an equivalence for each $s \in S$.

If $*$ is a one-element set, then $\mb{C}$ is isomorphic to the fiber of $\pi_0^\mb{C}$ over $*$. Let $i : \mb{C} \to \mb{C}^\amalg_\F$ be the resulting inclusion.
\begin{lem}
The functor
\[q : \mb{C}^\amalg_\F \to \Fun((\mb{C}^\amalg_\F)^\op, \Top) \os{i^*} \to \Fun(\mb{C}^\op, \Top)\]
is fully faithful, and its essential image is $\F_\mb{C}$.
\end{lem}
\begin{proof}
For each object $\ul{X} = (S, (X_s)) \in \mb{C}^\amalg_\F$, the right fibration classified by $q(\ul{X})$ is the source map 
\[\mb{C} \X_{\mb{C}^\amalg_\F} (\mb{C}^\amalg_\F)_{/ \ul{X}} \to \mb{C}.\]
As $s \in S$ varies, the maps 
\[j_s : \mb{C}_{/X_s} \to \mb{C} \X_{\mb{C}^\amalg_\F} (\mb{C}^\amalg_\F)_{/ \ul{X}}\]
with $j_s(Z \to X_s) = (Z \to X_s \to \ul{X})$ induce an isomorphism 
\[\coprod_{s \in S} \mb{C}_{/X_s} \cong \mb{C} \X_{\mb{C}^\amalg_\F} (\mb{C}^\amalg_\F)_{/ \ul{X}}.\]
\end{proof}
Let $\F_{\mb{C}}^\cart$ be the subcategory of $\pi_0^\mb{C}$-cartesian edges in $\F_\mb{C}$. We observe that $(\F_\mb{C}, \F_\mb{C}, \F_\mb{C}^\cart)$ is a disjunctive triple in the sense of \cite[Definition 5.2]{Bar14}. 
\begin{dfn}
We write 
\[\Add(\mb{C}) := A^{eff}(\F_{\mb{C}}, \F_{\mb{C}}, \F_{\mb{C}}^\cart)\] 
for the effective Burnside category of the disjunctive triple $(\F_{\mb{C}}, \F_{\mb{C}}, \F_{\mb{C}}^\cart)$ \cite[Definition 5.7]{Bar14}. It follows from \cite[5.8]{Bar14} and \cite[Proposition 4.3]{Bar14} that $\Add(\mb{C})$ is semiadditive.
\end{dfn} 

We can now state the main theorem of this section:
\begin{thm} \label{thm:addform}
The functor
\[\alpha : \mb{C} \os{i} \to \F_{\mb{C}} \to \Add(\mb{C})\]
exhibits $\Add(\mb{C})$ as the free semiadditive category on $\mb{C}$.
\end{thm}
\begin{proof}
First we characterize the mapping spaces in $\Add(\mb{C})$. Let $X, Y \in \mb{C}$. Then the mapping space $\Map_{\Add(\mb{C})}(X, Y)$ is the space of diagrams
\[\begin{tikzcd} [column sep = tiny] & X^{\amalg S} \ar{dl} \ar{dr} \\ X && Y \end{tikzcd}\]
up to equivalence over $X$ and $Y$. Since the automorphism group of $X^{\amalg S}$ over $X$ is just the symmetric group $\Si_S$, this space is equivalent to
\[\coprod_{n \geq 0} \Map_{\mb{C}}(X, Y)^n_{h \Si_n},\]
and on mapping spaces $\alpha$ induces the inclusion 
\[\alpha_{X, Y}: \Map_{\mb{C}}(X, Y) \to \Map_{\Add(\mb{C})}(X, Y)\]
which exhibits the target as the free commutative monoid on the source.

Now $\alpha$ extends to an additive functor 
\[\beta: \Add^p(\mb{C}) \to \Add(\mb{C}).\]
Since every object of $\Add(\mb{C})$ is a direct sum of objects of $\mb{C}$, $\beta$ is essentially surjective, and by Lemma \ref{lem:Addpmapsp}, for any objects $X, Y \in \mb{C}$, the induced map
\[\beta_{X, Y} : \Map_{\Add^p(\mb{C})}(X, Y) \to \Map_{\Add(\mb{C})}(X, Y)\]
is an equivalence. But this implies that $\beta_{X, Y}$ is an equivalence when $X, Y$ are direct sums of objects of $\mb{C}$, since in any additive category $\mb{M}$,
\[\Map_{\mb{M}}\left( \bigoplus_i X_i, \bigoplus_j Y_j \right) \simeq \prod_{i, j} \Map_{\mb{M}} (X_i, Y_j).\]
Thus $\beta$ is fully faithful, and we're done.
\end{proof}
In the case of most interest to us, when $\mb{C}$ is $A^{eff}(\F)$, this formula becomes substantially simpler:
\begin{prop}\label{prop:kappa}
The natural inclusion
\[\kappa : A^{eff}(\F) \to A^{eff}(\F_\F)\]
exhibits the target as the free semiadditive category on the source.
\end{prop}
\begin{proof}
$A^{eff}(\F_\F)$ is clearly semiadditive, and every object of $ A^{eff}(\F_\F) $ is a direct sum of objects of $ A^{eff}(\F) $. Moreover, due to the equivalence
\[ (BG)^n_{h \Si_n} \simeq B(G \wr \Si^n) ,\]
$ \kappa $ exhibits $ \Map_{A^{eff}(\F_\F)} (X,Y) $ as the free commutative monoid on $ \Map_{A^{eff}(\F)} (X,Y)$, for any objects $ X, Y \in A^{eff} (\F)$. Thus, by the same reasoning as in the proof of Theorem \ref{thm:addform}, we conclude.
\end{proof}
\section{Adjoining retractions in semiadditive $\iy$-categories}\label{sec:adjret}
In this section, we'll discuss a problem endemic to semiadditive category theory. Suppose $ \mb{C} $ is a semiadditive $\iy$-category, and 
\[ Y \os{i} \to X \os{r} \to Y \]
is a retraction diagram in $ \mb{C} $. Then $ ir $ is idempotent in $ \End(X) $. If $ \mb{C} $ were additive, then we would have a complementary idempotent $ 1-ir $ in $ \End(X) $, and if $ \mb{C} $ were moreover idempotent complete, then $ 1-ir $ would arise from a retraction
\[Y'\os{i'} \to X \os{r'} \to Y',\]
and we would obtain a direct sum decomposition
\[ X \simeq Y \oplus Y'. \]
However, if $ \mb{C} $ is merely semiadditive, then $ 1-ir $ need not exist, and therefore $ Y' $ need not exist, even if $ \mb{C} $ is idempotent complete.

As an example of this phenomenon, suppose that $G$ is a finite group and that $ \mb{B} $ is a full subcategory of the category $ \F_G $ of finite $G$-sets which is closed under pullbacks and disjoint unions. Suppose furthermore that $ \mb{B} $ contains an object of the form $ O \amalg S $, where $O$ is an orbit and $S$ is some other finite $ G $ -set, and suppose that $ \mb{B} $ contains $ O $ but not $ S $. For instance, $\mb{B}$ might be the category of finite $G$-sets with at least one element on which $G$ acts trivially. Then $A^{eff}(\mb{B})$ is semiadditive, and the retraction of $O \amalg S$ onto $O$ in $A^{eff}(\mb{B})$ gives rise to an idempotent on $O \amalg S$. But the complementary idempotent is nowhere to be found.

We claim that $A^{eff}(\F_\F)$ is precisely this kind of example. Note that the endomorphism
\[ \begin{tikzcd} [column sep=tiny]
& \{a \} \amalg \{b\}\ar{dl} \ar{dr}  \\
\{a, b\} && \{a, b\}
\end{tikzcd} \]

of $\{a,b\} \in A^{eff}(\F_\F)$ is idempotent. But it's easy to see that no complementary idempotent exists. What should the remaining summand be? Or, otherwise put, where's the total cofiber of 

\[\begin{tikzcd} \emptyset \ar{r} \ar{d} & \{a\} \ar{d} \\
\{b\} \ar{r} & \{a, b\}? \end{tikzcd}\]

In the following construction, we will universally adjoin such complementary idempotents. We'll define a \emph{complementable semiadditive category} to be, roughly speaking, a semiadditive categories in which all idempotents have complementary summands, and we'll characterize the category obtained by adjoining complements to $A^{eff}(\F_\F)$. This will be the effective Burnside category of finite sets and surjections, $A^{eff}(\F_{\Fs})$, which will play an important role in the statement of the main theorem and the remainder of this paper.

%
%

As in the introduction, let $\Fs$ be the category of finite sets and surjective maps, and let $\Fs^{\leq n} $ be the full subcategory spanned by the sets of cardinality at most $n$. The first important fact about $\Fs$ is as follows:
\begin{lem} \label{lem:surjpbs}
	$\F_{\Fs}$ admits pullbacks.
\end{lem}
\begin{proof}
	It suffices to show that any diagram $\La^2_2 \to \Fs$ admits a pullback in $\F_{\Fs}$, since pullbacks distribute over coproducts, which are disjoint, in any category of the form $\F_{\mb{C}}$. But it's easy to see that the diagram
	 \[ \begin{tikzcd} \coprod_{U \inc Y \X_X Z, U \twoheadrightarrow Y, U \twoheadrightarrow Z} \ar{r} \ar{d} & Y \ar{d} \\ Z \ar{r} & X \end{tikzcd} \]
	 is a pullback square, where the coproduct is taken over subsets of the set pullback $Y \X_X Z$ which surject onto both $Y$ and $Z$ under the natural projections. Indeed, if 
	 \[\begin{tikzcd} V \ar{r} \ar{d} & Y \ar{d} \\ Z \ar{r} & X  \end{tikzcd} \]
	 is a commutative square of surjections, then $V$ surjects onto a unique such subset.
\end{proof}
\begin{dfn}\label{def:phi}

Define a functor 
\[ \phi' : \F \to \F_{\Fs} \]
by setting, for any finite set $S$,
\[\phi'(S) = \coprod_{U \inc S} U \]
and, for each morphism $f: S \to T$, letting the component of $\phi(f)$ on $U \inc X$ be the natural surjection
\[f_U : U \to \im U.\]
By the universal property of $\F_\F$, $\phi'$ extends uniquely to a coproduct-preserving functor
\[\phi : \F_\F \to \F_{\Fs}.\]
\end{dfn}
The following is an easy observation:
\begin{lem}
	$\phi$ is right adjoint to the functor 
	\[U : \F_{\Fs} \to \F_\F\]
	which is given by the identity on objects and the inclusion of setwise surjective maps into all maps on morphisms.
\end{lem}
In particular,
\begin{cor}
	$\phi$ preserves pullbacks.
\end{cor}
This corollary may also be established by direct inspection using the description of the pullbacks of $\F_{\Fs}$ given in Lemma \ref{lem:surjpbs}. We deduce that $\phi$ gives rise to a functor between effective Burnside categories, which we'll abusively denote
\[ \phi : A^{eff}(\F_\F) \to A^{eff}(\F_{\Fs}).\]
We'll show that $\phi$ formally adjoins certain total cofibers to $A^{eff}(\F_\F)$.

%
%
\begin{rec}\label{rec:nonabder}
Let $\mb{C}$ be a semiadditive category, and denote 
\[\Psh^\oplus(\mb{C}) = \Fun^\X(\mb{C}^{\op}, \Top), \]
the category of \emph{additive presheaves} on $\mb{C}$, which are functors $\mb{C}^\op \to \Top$ which carry direct sums to products. Such functors factor uniquely through the forgetful functor $\CMon \to \Top$ by Remark \ref{rem:cmonext}, and so we may equivalently define
\[ \Psh^\oplus(\mb{C}) = \Fun^\oplus(\mb{C}^\op, \CMon). \]
$ \Psh^\oplus(\mb{C}) $ is freely generated as a semiadditive category under all colimits, and as a category under sifted colimits, by $\mb{C}$. It is also known as the \emph{nonabelian derived category} of $\mb{C}$. For a general reference on this object, see \cite[Definition 5.5.8.8]{HTT} and the ensuing discussion.
\end{rec}
The content of the next lemma is that each representable presheaf in \newline $\Psh^\oplus(A^{eff}(\F_{\Fs}))$ has a filtration by split monomorphisms - morphisms that ``should" be direct summand inclusions.
\begin{lem}
	Let $S$ be a finite set, let $h_S \in \Psh^\oplus(A^{eff}(\F_{\F}))$ be the functor represented by $S$, and for any $k$ with $0 \leq k \leq |S|$, let $h_S^{\leq k}$ denote the full subfunctor of $h_S$ whose value on $X \in A^{eff}(\F_{\F})$ is the subspace of $h_S(X)$ spanned by those diagrams
	\[ \begin{tikzcd}[column sep = small] & Y \ar{dr} \ar{dl} \\ X && S \end{tikzcd} \]
	for which the image in $S$ of every component of $Y$ has cardinality $\leq k$. Then the natural inclusion
	\[ J_k : h_S^{\leq k - 1} \to h_S^{\leq k} \]
	admits a retraction for each $k$ with $0 < k \leq |S|$.
	
\end{lem}
\begin{proof}
	This is easiest to see at the level of right fibrations. The right fibration corresponding to $h_S$ is the overcategory $A^{eff}(\F_\F)_{/S}$, whose $n$-simplices are diagrams
	\[ \delta : \twa(\Delta^{n + 1})^{\op} \to \F_\F, \]
	where $\twa$ denotes the twisted arrow category (see \cite[Example 2.6]{Bar14}), such that all squares are pullbacks and $\delta(n \to n) = S$. The right fibration corresponding to $h_S^{\leq k}$ is the simplicial subset 
	\[A^{eff}(\F_\F)_{/S}^{\leq k} \inc A^{eff}(\F_\F)_{/S}\]  whose $n$-simplices $\delta$ satisfy the additional condition that the image in $S$ of each component of $\delta((n - 1) \to n)$ has cardinality $\leq k$. $J_k$ is the natural inclusion
	\[ A^{eff}(\F_\F)_{/S}^{\leq k - 1} \inc A^{eff}(\F_\F)_{/S}^{\leq k},  \]
	and the retraction, $R_k$, takes an $n$-simplex $\sigma$ of $A^{eff}(\F_\F)_{/S}^{\leq k}$ to $R_k(\sigma)$, where
	\begin{itemize}
		\item the restriction of $R_k(\sigma)$ to $\twa(\Delta^{[0, \cdots, n- 1]})^\op \inc \twa(\D^n)^{\op}$ is identical to that of $\sigma$, and
		\item for each $m$ with $0 \leq m \leq n-1$, $R_k(\sigma)(m \to n)$ is the subobject of components of $\sigma(m \to n)$ whose image in $S$ has cardinality at most $k - 1$.
	\end{itemize}
	It's easy to see that $R_k(\sigma)$ really is an $n$-simplex of $A^{eff}(\F_\F)_{/S}^{\leq k - 1}$ - that is, squares remain pullbacks - and that $R_k$ is left inverse to $J_k$.
\end{proof}
\begin{lem} \label{lem:cofJk}
	The cofiber of $J_k$ can be identified with 
	\[ \bigoplus_{U \inc S, \, |U| = k} \phi^* \hbar_U \]
	where $\hbar_U \in \Psh^\oplus(A^{eff}(\F_{\Fs}))$ is the presheaf represented by $U$.
\end{lem}
\begin{proof}
	Clearly
	\[ \phi_! h_S \simeq \bigoplus_{U \inc S} \hbar_U. \]
	Composing with the projection to $\bigoplus_{U \inc S, \, |U| = k + 1} \hbar_U$ and adjointing over, we get a map
	\[ h_S \to \bigoplus_{U \inc S, \, |U| = k + 1} \phi^* \hbar_U \]
	Composing further with the inclusion $h_S^{\leq k + 1} \to h_S$ gives the required map 
	\[ L : h_S^{\leq k + 1} \to \bigoplus_{U \inc S \, |U| = k + 1} \phi^* \hbar_U. \]
	At the level of objects, $L$ and $J_{k + 1}$ give the cofiber sequence of commutative monoids
	\[ h_S^{\leq k}(X) \to h_S^{\leq k + 1}(X) \to \] \[  \left \{ \begin{tikzcd} [column sep = small] &Y \ar{dr}{f}  \ar{dl} \\ X && S \end{tikzcd} \, \middle| \, |f(Y')| = k \text{ for every component } Y' \text{ of } Y \right\}. \]
\end{proof}
\begin{lem}
	For each $U \in F_{\Fs}$, the counit map
	\[ \eps_U : \phi_! \phi^* U \to U \]
	is an equivalence. Therefore, extending by colimits, the counit map
	\[ \eps : \phi_! \phi^* \to \id_{A^{eff}(\F_{\Fs})} \]
	is an equivalence of functors.
\end{lem}
\begin{proof}
	Let $U$ be a finite set of cardinality $n$. It's then clear that
	\[h_U^{\leq n - 1} \simeq \colim{V \in \mb{Cube}^U \setminus \{U\}} h_V,\]
	where $\mb{Cube}^U$ is the poset of subsets of $U$. Thus we have
	\[ \phi_! (h_U^{\leq n - 1}) \simeq \colim{V \in \mb{Cube}^U \setminus \{U\}} \left( \bigoplus_{W \inc V} \hbar_W \right) \simeq \bigoplus_{V \subsetneq U} \hbar_V, \]
	and a cofiber sequence
	\[ \phi_! (h_U^{\leq n - 1}) \os{\phi_! J_n} \to \phi_! h_U \to \hbar_U, \]
	where the right hand map is the projection 
	\[ \phi_! h_U \simeq \bigoplus_{V \inc U} \hbar_V \to \hbar_U. \]
	But applying $\phi_!$ to the cofiber sequence obtained in Lemma \ref{lem:cofJk} shows that the cofiber of $\phi_! J_n$ can be identified with $\phi_! \phi^* \hbar_U$. This gives the result.
\end{proof}
We deduce from this that $\phi^*$ is fully faithful. This is the backdrop to the major result of this section: that $A^{eff}(\F_{\Fs})$ is the localization of $A^{eff}(\F_\F)$ into ``complementable semiadditive categories".
\begin{dfn}
	A semiadditive category $\mb{C}$ is called \emph{complementable} if whenever $f : X \to Y$ is a morphism in $\mb{C}$ admitting a retraction $r : Y \to X$, the cofiber $g : Y \to Z$ of $f$ exists, and the diagram
	\[ \begin{tikzcd} Y \ar{r}{g} \ar{d}{r} & Z \\ X \end{tikzcd} \]
	is a product diagram.
\end{dfn}
\begin{exa}
	Any idempotent-complete additive category $\mb{C}$ is complementable.
\end{exa}
\begin{proof}
	The homotopy category $\ho(\mb{C})$ is an idempotent complete additive category, and we may define $Z \in \ho(\mb{C})$ to be the retract of $Y$ corresponding to the idempotent $\id - fr$. Then there is a product diagram 
	\[ \begin{tikzcd} Y \ar{r}{g} \ar{d}{r} & Z \\ X \end{tikzcd} \]
	in $\ho(\mb{C})$, which lifts to a product diagram in $\mb{C}$. It remains to show that $X \os{f} \to Y \os{g} \to Z$ is a cofiber sequence in $\mb{C}$. This follows from the fact that $f$ is homotopic to the summand inclusion $X \to X \oplus Z \simeq Y$, and $g$ is homotopic to the projection $Y \simeq X \oplus Z \to Z$.
\end{proof}
\begin{lem}
	Let $\mb{C}$ be a complementable semiadditive category which is presentable and let $\psi : A^{eff}(\F_\F) \to \mb{C}$ be an additive functor. Abusively denoting the colimit-preserving extension
	\[ \psi : \Psh^\oplus(A^{eff}(\F_\F)) \to \mb{C}, \]
	we claim that for each finite set $S$, $\psi$ takes the map
	\[ B_S :  h_S \to \bigoplus_{U \inc S} \phi^* \hbar_U \]
	to an equivalence.
\end{lem}
\begin{proof}
	We'll prove simply by induction on $k$ with $0 \leq k \leq |S|$ that $\psi(B_S^{\leq k})$ is an equivalence, where 
	\[ B_S^{\leq k} : h_S^{\leq k} \to \bigoplus_{U \inc S, \, |U| \leq k} \hbar_U \]
	is the natural map. Indeed, $B_S^{\leq 0}$ is already an equivalence, and since $\psi$ must take the diagram
	\[ \begin{tikzcd} h_S^{\leq k} \ar{r}{L} \ar{d}{R_k} & \bigoplus_{U \inc S, \, |U| = k} \hbar_U \\ h_S^{\leq k - 1} \end{tikzcd} \]
	to a product diagram, the conclusion follows.
\end{proof}
Since $\phi_! B_S$ is also an equivalence, it follows that any object in the essential image of $\phi^*$ is local for the collection of morphisms $ \mb{B} := \{B_S \, | \, S \text{ a finite set}\}.$ In particular, $B_S$ exhibits $\bigoplus_{U \inc S} \phi^*\hbar_U$ as the $\mb{B}$-localization of $h_S$. Conversely, since $\mb{B}$-localization preserves colimits, all $\mb{B}$-local objects belong to $\im(\phi^*)$. We deduce:
\begin{thm}\label{thm:phi}
	$\phi_!$ exhibits $\Psh^{\oplus}(A^{eff}(F_{\Fs}))$ as the $\mb{B}$-localization of \newline $\Psh^{\oplus}(A^{eff}(\F_\F))$. As a consequence, if $\mb{C}$ is a presentable complementable semiadditive category, the restriction
	\[ \phi^* : \Fun^{\oplus}(A^{eff}(\F_{\Fs}), \mb{C}) \to \Fun^{\oplus}(A^{eff}(\F_\F), \mb{C}) \]
	is an equivalence of categories. This holds in particular if $\mb{C}$ is a presentable (thus idempotent complete) additive category.
\end{thm}
\begin{proof}
	At this point, only the second sentence requires proof; but it follows immediately from the diagram
	\[ \begin{tikzcd} \Fun^L(\Psh^{\oplus}(A^{eff}(\F_{\Fs})), \mb{C}) \ar{r}[above]{\phi^*}[below]{\sim} \ar{d}{\sim} & \Fun^L(\Psh^\oplus(A^{eff}(\F_\F)), \mb{C}) \ar{d}{\sim} \\ \Fun^{\oplus}(A^{eff}(\F_{\Fs}), \mb{C}) \ar{r}{\phi^*} & \Fun^{\oplus}(A^{eff}(\F_\F), \mb{C}), \end{tikzcd} \]
	in which the other three maps are equivalences.
\end{proof}
\begin{rem}
	We haven't actually proved that $A^{eff}(\F_{\Fs})$ is complementable - we only need the fact that it's constructed by adjoining \emph{some} complements to $A^{eff}(\F_\F)$ - but we strongly suspect that it is.
\end{rem}

\section{Calculus and the Burnside category}\label{sec:calcburn}
In this section we'll assume that the reader is familiar with the basic definitions and results of Goodwillie's functor calculus as laid out, for instance, in \cite{Goo03}, the original and still best source, or in the notes from the 2012 Talbot workshop, or in \cite[\S 6.1]{HA}. 

Let $\mb{E}$ be a stable $\iy$-category. We'd like to progressively reduce the theory of polynomial functors $\Sp^\omega \to E$ to something more combinatorial. Our first lemma in this direction asserts that polynomial functors on the category of connective finite spectra extend uniquely, in the strongest possible way, to all finite spectra. This result is originally due to Lukas Brantner, and the line of argument is significantly indebted to a conversation between us.

\begin{prop}[Brantner]\label{prop:connsp}
	Let $\iota : \Sp^\omega_{\geq 0} \to \Sp^\omega$ be the inclusion. For each $n$, the restriction functor
	\[ \iota^*_n: \Fun^{n-exc}(\Sp^\omega, \mb{E}) \to \Fun^{n-exc}(\Sp^\omega_{\geq 0}, \mb{C}) \]
	is an equivalence of categories.
\end{prop}

By comparison, the functor $\iota^* : \Fun(\Sp^\omega, \mb{E}) \to \Fun(\Sp^\omega_{\geq 0}, \mb{E})$ is very far from even being conservative - as we know well, polynomial functors are very rigid compared to general functors.

\begin{proof}
	We'll use a recollement argument and \cite[Proposition 8]{BGc} using the recollement formed by the reflective and coreflective subcategory
	\[  \Fun^{(n-1)-exc}(\mb{S}, \mb{E}) \inc \Fun^{n-exc}(\mb{S}, \mb{E}) \]
	where $\mb{S}$ is either $\Sp^\omega_{\geq 0}$ or $\Sp^\omega$. Denote $\Fun^{(n-1)-exc}(\Sp^\omega, \mb{E})$ by $\mb{U}$ and \newline $\Fun^{(n-1)-exc}(\Sp_{\geq 0}^{\omega}, \mb{E})$ by $\mb{U}_{\geq 0}$. Denote the categories of objects of $\Fun^{n-exc}(\mb{S}, \mb{E})$ which are left resp. right orthogonal to $\mb{U}_{(\geq 0)}$ by $\mb{Z}^\vee_{(\geq 0)}$ resp. $\mb{Z}^{\wedge}_{(\geq 0)}$. By definition, $\mb{Z}^\vee_{(\geq 0)}$ is the category of $n$-homogeneous functors $\mb{S} \to \mb{E}$. 

	By induction on $n$ and \cite[Proposition 8]{BGc}, to prove the proposition it's enough to show that
	\[ \iota_n^*(\mb{Z}^{\vee}_{\Sp^{\omega}}) \simeq \mb{Z}^{\vee}_{\Sp_{\geq 0}^{\omega}}, \hspace{20pt}
	 \iota_n^*(\mb{Z}^{\wedge}_{\Sp^{\omega}}) \inc \mb{Z}^{\wedge}_{\Sp_{\geq 0}^{\omega}}.\] 
	
	By \cite[Proposition 6.1.4.14]{HA}, we have equivalences
	
	\begin{align*} \Fun^{n-hmg}(\mb{S}, \mb{E}) & \simeq \mb{SymFun}^n_{\text{lin}}(\mb{S}, \mb{E}) \\
		& \simeq \Fun^{1-exc}(\mb{S}^{\otimes n}_{h \Si_n}, \mb{E}), \end{align*}
	the category of symmetric multilinear functors, whence
	\begin{align*}
	\Fun^{n-hmg}(\mb{S}, \mb{E}) & \simeq \Fun^{1-exc}(\mb{S}_{h \Si_n}, \mb{E}) \\
	& \simeq \Fun^{1-exc}(\mb{S}, \mb{E}^{h \Si_n}) \\
	& \simeq \mb{E}^{h \Si_n}
	\end{align*}
	for the trivial $\Si_n$-actions on $\mb{S}$ and $\mb{E}$. Indeed, since $\Sp^\omega_{\geq 0}$ is freely generated by $\bb{S}$ as an additive category with finite colimits, and $\Sp^\omega$ is freely generated by $\bb{S}$ as a stable category, both are tensor-idempotent, and we have
	\[ \Fun^{1-exc}(\mb{S}, \mb{E}) \simeq \mb{E} \]
	whenever $\mb{E}$ is stable. Since this chain of equivalences is compatible with restriction along $\iota$, this shows that 
	\[ \iota_n^*(\mb{Z}^{\vee}) \simeq \mb{Z}^{\vee}_{\geq 0} .\]
	For the inclusion
	\[ \iota_n^*(\mb{Z}^{\wedge}) \inc \mb{Z}^{\wedge}_{\geq 0}, \]
	we let $L_n$ be the left adjoint of $\iota_n^*$. This is left Kan extension $\iota_!$ along $\iota$ followed by $n$-excisivization:
	\[ L_n = P_n\iota_!. \] 
	Now it suffices to show that
	\[ L_n(\mb{U}_{\geq 0}) \inc \mb{U}, \] 
	since this implies that if $X \in \mb{Z}^{\wedge}$ and $Y \in \mb{U}_{\geq 0}$, then
	\[ \Map(Y, \iota_n^* Z) \simeq \Map(L_n Y, Z) \simeq *. \]
	Since $\iota$ itself has a right adjoint, the connective cover functor $\tau_{\geq 0}$, $\iota_!$ is just precomposition with $\tau_{\geq 0}$. This, along with the next set of lemmas, will show that if $F$ is $(n-1)$-excisive then $L_n F$ is $(n - 1)$-excisive:
	
	\begin{lem}
		Suppose $F : \Sp \to \mb{E}$ is a functor whose restriction to $\Sp_{\geq 0}$ is $n$-excisive. Then the unit of the $n$-excisivization
		\[ \eta : F \to P_n F \]
		is an equivalence on objects of $\Sp_{\geq 0}$.
	\end{lem}
	\begin{proof}
		This is clear from Goodwillie's original construction of $P_n F$, which is described in sufficient generality in \cite[Construction 6.1.1.27]{HA}.
	\end{proof}
	We deduce that
	\[ \iota_n^* L_n \simeq \id. \]
	\begin{lem}
		If $F$ is $n$-excisive and $F$ evaluates to zero on connective spectra, then $F = 0$.
	\end{lem}
	\begin{proof}
		This means that
		\[ \text{cr}_nF(\bb{S}, \cdots, \bb{S}) = 0, \]
		and so the $n$th derivative of $F$ is zero and $F$ is actually $(n-1)$-excisive. By induction down the Taylor tower, $F = 0$.
	\end{proof}
	\begin{lem}\label{lem:resn-1exc}
		If the restriction of $F$ to $\Sp_{\geq 0}$ is $(n - 1)$-excisive, then $P_n F$ is $(n - 1)$-excisive.
	\end{lem}
	\begin{proof}
		The fiber of $P_nF \to P_{n - 1}F$ is zero on connective spectra, and therefore it is zero. Thus $P_nF$ is $(n-1)$-excisive.
	\end{proof}
	The hypotheses of Lemma \ref{lem:resn-1exc} apply in particular to $L_n F$ when $F \in \mb{U}_{\geq 0}$. This completes the proof of Proposition \ref{prop:connsp}.
	
\end{proof}
Suppose now that $\mb{E}$ is presentable. Let $A(\F)$ be the Burnside category of finite sets, or equivalently, the full subcategory of $\Sp$ spanned by spectra equivalent to finite direct sums of copies of $\bb{S}$. When discussing objects of $A(\F)$, we will use $\angs{n}$ to denote a chosen $n$-element set. Let $\Im$ denote the inclusion $A(\F) \to \Sp$. Then $\Im$ exhibits $\Sp_{\geq 0}$ as the nonabelian derived category of $A(\F)$. Therefore:

\begin{cor}\label{cor:resaf}
For each $n$, the restriction functor
\[ \Im^*_n : \Fun^{n-exc}(\Sp_{\geq 0}, \mb{E}) \to \Fun(A(\F), \mb{E}) \]
is fully faithful.
\end{cor}

\begin{proof}
	Indeed, if $\Fun^\Si(\Sp_{\geq 0}, \mb{E})$ is the category of sifted-colimit-preserving functors from $\Sp$ to $\mb{E}$, then the restriction
	\[ \Im^*_\Si : \Fun^\Si(\Sp_{\geq 0}, \mb{E}) \to \Fun(A(\F), \mb{E}) \]
	is an equivalence. But it follows from Goodwillie's classification of homogeneous functors (see e.g. the proof of \cite[Corollary 6.1.4.15]{HA}, since Goodwillie's papers can be a little hard to track down these days), together with induction on $n$ and the fact that tensor power functors preserve sifted colimits, that $n$-excisive functors all preserve sifted colimits.
\end{proof}

\begin{prop}\label{prop:nexcessim}
	The essential image of $\Fun^{n-exc}(\Sp_{\geq 0}, \mb{E})$ in $\Fun(A(\F), \mb{E})$ is the category $\Fun^{\leq n}(A(\F), \mb{E})$ spanned by those functors $F$ which, for each finite set $S$ of cardinality $|S| > n$, map the cube of projections (a.k.a. contravariant injections or inert maps)
	\[\varrho_S : (\mb{Cube}^S)^\op \to A(\F) \hspace{50pt} U \mapsto U\]
	to a cartesian cube in $E$. Rephrased, this is the condition that for each $m > n$, the cross effect
	\[ \text{cr}_SF(\angs{1}, \angs{1}, \cdots, \angs{1}) = 0.\]
	We will call objects of $\Fun^{\leq n}(A(\F), \mb{E})$ \emph{degree $n$ functors}.
\end{prop}
\begin{proof}
	Since all $m$th order cross effects of $n$-excisive functors vanish for $m > n$, it's clear that the image of $\Im_n^*$ is contained in $\Fun^{\leq n}(A(\F), \mb{E})$. To prove the converse, we need to show that the left Kan extension along $\Im$ of any $F : A(\F) \to \mb{E}$ of degree $n$ is $n$-excisive. For this, we note that $G \in \Fun^{\leq n}(\Sp_{\geq 0}, \mb{E})$ is $n$-excisive if and only if the cross effects
	\[ \text{cr}_{\angs{m}} G(X_1, X_2, \cdots, X_m), \]
	are zero whenever $m > n$. On the other hand, we may write
	\[ X_i = \colim{\D^\op} X_{i, \bullet} \]
	where $X_{i, \bullet} : \D^\op \to A(\F)$ is a diagram with image in $A(\F)$. Since $G$ preserves sifted colimits, we have
	\[\text{cr}_{\angs{m}} G(X_1, X_2, \cdots, X_m) \simeq \colim{\D^{\op}}\text{cr}_{\angs{m}}(X_{1, \bullet}, X_{2, \bullet}, \cdots, X_{n, \bullet}.)  \]
	So it suffices to prove that if $F \in \Fun^{\leq n}(A(\F), \mb{E})$, then
	\[ \text{cr}_{\angs{m}}F(S_1, S_2, \cdots, S_m) = 0 \]
	whenever $m > n$ and $S_1, S_2, \cdots S_m$ are objects of $A(\F)$. 
	
	Now suppose that $S$ is a finite set and $\phi : S \to T$ is a surjective map of sets. We get a functor
	\[ \phi^* : (\mb{Cube}^T)^{\op} \to (\mb{Cube}^S )^{\op}\]
	by letting $\psi(U)$ be the preimage of $U$ under $\phi$. We claim that if $F : A(\F) \to \mb{E}$ is degree $n$, then $F \varrho_S \phi^*$ is a cartesian $T$-cube whenever $|T| > n$; if we can establish this, we can conclude, because $F \varrho_S \phi^*$ is the cube whose total fiber is the cross effect
	\[ \text{cr}_T F((\phi^{-1}(t))_{t \in T}). \]
	 We'll prove this statement by induction on $|S|$ and $|S| - |T|$, by representing $F \varrho_S \phi^*$ as the concatenation of two $T$-cubes we know to be cartesian by the induction hypothesis. This involves some moderately hard combinatorics.
	
	If $|S| = |T|$, then $\phi^*$ is an isomorphism and the conclusion is clear. Moreover, if $|S| = n + 1$, then this is always the case. Otherwise, we may choose some $t \in T$ such that $|\phi^{-1}(T)| > 1$, and pick some $s \in \phi^{-1}(T)$. By induction, 
	\[ F \varrho_{S \setminus \{s\}} (\phi|_{S \setminus \{s\}})^* : (\mb{Cube}^T)^{\op} \to \mb{E} \]
	is cartesian. Now define
	\[ \psi : (\mb{Cube}^T)^{\op} \to (\mb{Cube}^S)^{\op} \]
	by
	\[ \psi(U) = \begin{cases} \phi^*(U) & t \in U \\ \phi^*(U \cup \{t\}) \setminus \{s\} & t \notin U. \end{cases} \]
	We claim that  
	\[ F \varrho_S \psi : (\mb{Cube}^{T})^{\op} \to \mb{E} \]
	is cartesian. To see this, let
	\[ T' = T \cup \{s\} \]
	and define $\lambda : S \to T'$ by
	\[ \lambda(x) = \begin{cases}
	\phi(x) & x \neq s \\
	s & x = s.
	\end{cases}  \]
	By induction, $F \varrho_S \lambda^*$ is cartesian. Moreover, if
	\[ T'' = T' \setminus \{t\}, \, \, S'' = S \setminus (\phi^{-1}(t) \setminus \{s\}), \]
	then the restriction of $F \varrho_S \lambda^*$ to the face $(\mb{Cube}^{T''})^\op$ of $(\mb{Cube}^{T'})^\op$ is naturally identified with
	\[ F \varrho_{S''} (\phi|_{S''})^* : (\mb{Cube}^{T''})^{\op} \to \mb{E}, \]
	and is thus also cartesian by induction. Thus the opposite face of $F \varrho_S \lambda^*$ is also cartesian. But this face can be identified with $F \varrho_S \psi$.
	
	Now 
	\[ \psi|_{(\mb{Cube}^{T \setminus \{t\}})^{\op}} : (\mb{Cube}^{T \setminus \{t\}})^{\op} \to (\mb{Cube}^{S})^{\op} \]
	has the same image as
	\begin{align*} (\phi|_{S \setminus \{s\}})^*|_{(\mb{Cube}^{T})^{\op} \setminus (\mb{Cube}^{T \setminus t})^{\op}} & : (\mb{Cube}^{T})^{\op} \setminus (\mb{Cube}^{T \setminus \{t\}})^{\op} \\& \to (\mb{Cube}^{S \setminus \{s \}})^{\op} \\& \to (\mb{Cube}^{S})^{\op}.\end{align*}
	Gluing $\psi$ and $(\phi|_{S \setminus \{s\}})^{*}$ along this common face gives $\phi^*$, and applying $F \varrho_S$ represents $F \varrho_S \phi^*$ as the composite of two $T$-cubes proven to be cartesian, which gives the result.
	
\end{proof}
\section{Analytic continuation of polynomials}\label{sec:anacon}
In this section, we'll prove the following possibly surprising result:
\begin{thm}\label{thm:anacon}
	Let $\mb{E}$ be a stable category. Let $\Fun^{\leq n}(A(\F), \mb{E})$ be as in Proposition \ref{prop:nexcessim}, and let \newline $\Fun^{\leq n}(A^{eff}(\F), \mb{E})$ be the full subcategory of $\Fun(A^{eff}(\F), \mb{E})$ defined by the same condition on vanishing of cross effects. Then restriction along the additive completion functor
	\[ R_n : \Fun^{\leq n}(A(\F), \mb{E}) \to \Fun^{\leq n}(A^{eff}(\F), \mb{E}) \]
	is an equivalence of categories.
\end{thm}
Of course, when $n = 1$, this follows immediately from the universal property of the additive completion functor. That's exactly what makes this statement look strange: constructions defined by universal properties generally deliver neither more nor less than they promise, but the additive completion appears to be bringing unexpected gifts.

We think of this statement as analogous to the following elementary result from analysis: analytic continuation exists and is unique for polynomials. That is, a polynomial function defined on the right half plane $\{s \in \bb{C} \, | \, \Re \, s \geq 0\}$ extends uniquely to a polynomial function defined on the whole of $\bb{C}$.

We thank Akhil Mathew for insights critical to the proof of this theorem. The proof follows the plan of the proof of Proposition \ref{prop:connsp} very closely, and will refer back to that proof at several points; common generalizations of these theorems and much more will appear in future joint work.

\begin{proof}
	First suppose that $\mb{E}$ is presentable. Note that the proof of Proposition \ref{prop:nexcessim} shows verbatim that the restriction
	\[A_n : \Fun^{n-exc}(\CMon, \mb{E}) \to \Fun^{\leq n}(A^{eff}(\F), \mb{E}) \]
	is an equivalence of categories. We will use a recollement argument and induction on $n$ to show that $A_n$ is an equivalence by induction on $n$.
	
	Let
	\[ \mb{U} = \Fun^{(n-1)-exc}(\Sp_{\geq 0}, \mb{E}) \inc \Fun^{n-exc}(\Sp_{\geq 0}, \mb{E}), \]
	and let $\mb{Z}^{\vee}$, $\mb{Z}^{\wedge}$ be the subcategories left resp. right orthogonal to $\mb{U}$. Similarly, define
	\[ \mb{U}^{eff} = \Fun^{(n-1)-exc}(\CMon, \mb{E}) \inc \Fun^{n-exc}(\CMon, \mb{E}), \]
	and let $\mb{Z}^{eff, \vee}$ and $\mb{Z}^{eff, \wedge}$ be its left and right orthogonal subcategories. By the induction hypothesis,
	\[ (A_n)|_{\mb{U}} : \mb{U} \to \mb{U}^{eff} \]
	is an equivalence. By \cite[Proposition 8]{BGc}, in order to show that an $A_n$ is an equivalence, it suffices to show that
	\[ A_n(\mb{Z}^{\vee}) \simeq \mb{Z}^{eff, \vee}, \hspace{20pt}
	A_n(\mb{Z}^{\wedge}) \inc \mb{Z}^{eff, \wedge}.\] 
	For the first equivalence, we note that $\mb{Z}^{\vee}$ is the category of $n$-homogeneous functors $\Sp_{\geq 0} \to \mb{E}$ and $\mb{Z}^{eff, \vee}$ is the category of $n$-homogeneous functors $\CMon \to \mb{E}$. As in the proof of Proposition \ref{prop:connsp}, the usual classification goes through and both categories are identified with $\mb{E}^{h \Si_n}$ by the $n$th cross effect construction.
	
	For the second inclusion, as in the proof of Proposition \ref{prop:connsp}, we consider the left adjoint $L_n$ of $R_n'$; it suffices to show that
	\[ L_n(\mb{U}^{eff}) \inc \mb{U}. \]
	
	 If $GC$ is the group completion functor from $\CMon$ to $\Sp_{\geq 0}$, then $L_n$ is left Kan extension along $GC$ followed by $n$-excisivization:
	\[ L_n = P_n (GC)_!. \]
	Since $GC$ has a right adjoint, namely the inclusion $\varpropto$ of $\Sp^{\geq 0}$ as a full subcategory of $\CMon$, we can identify $(GC)_!$ with $\varpropto^*$. But since $\varpropto$ itself preserves colimits, $\varpropto^*$ preserves $k$-excisive functors for any $k$; thus 
	\[ L_n = \varpropto^*. \]
	In particular, we can take $k = n - 1$, and the conclusion follows.
	
	We have now proved Theorem \ref{thm:anacon} in the case where $\mb{E}$ is presentable. To deduce the result for a general stable category $\mb{E}$, we may replace $\mb{E}$ by the presentable category $\text{Ind}(\mb{E})$, which has $\mb{E}$ as a full subcategory. If
	\[ F: A(\F) \to \text{Ind}(\mb{E}) \]
	is a degree-$n$ functor, then clearly $F \in \Fun^{\leq n}(A(\F), \mb{E})$ if and only if $A_nF \in \Fun^{\leq n}(A^{eff}(\F), \mb{E}).$
\end{proof}

\section{Synthesis}\label{sec:synth}
In this brief concluding section, we declare our hand and chain the functors we've constructed to produce the main equivalence.

\begin{thm}\label{thm:main}
	If $\mb{E}$ is stable and presentable, there is an explicit equivalence of categories from $ \Mack(\Fs^{\leq n}, \mb{E})$ to $\Fun^{n-exc}(\Sp^{\omega}, \mb{E})$.
\end{thm}
\begin{proof}
	The equivalence is written as a composite as follows:
	\begin{align*}
	\Mack(\Fs^{\leq n}, \mb{E}) & \os{\Xi^{\Fs^{\leq n}}} \to \Mack^{\leq n}(\Fs, \mb{E}) \\
	& \os{\phi^*} \to \Fun^{\oplus, \leq n}(A^{eff}(\F_\F), \mb{E}) \\
	& \os{\kappa^*} \to \Fun^{\leq n}(A^{eff}(\F), \mb{E}) \\
	& \os{(R_n)^{-1}} \to \Fun^{\leq n}(A(\F), \mb{E}) \\
	& \os{(\Im_n^*)^{-1}} \to \Fun^{n-exc}(\Sp^{\omega}_{\geq 0}, \mb{E}) \\
	& \os{(\iota^*)^{-1}} \to \Fun^{n-exc}(\Sp^{\omega}, \mb{E}).
	\end{align*}
	A word of clarification is necessary, because some of these categories are minor variants appearing for the first time. $\Mack^{\leq n}(A^{eff}(\F_{\Fs}), \mb{E})$ is the full subcategory of \[\Mack(\Fs, \mb{E}) := \Fun^{\oplus}(A^{eff}(\F_{\Fs}), \mb{E})\]
	spanned by those functors whose value on $U \in \Fs$ is zero whenever $|U| > n$. The equivalence $\Xi^{\Fs^{\leq n}}$ is a case of \cite[Corollary 2.33]{Gla15a}. $\Fun^{\oplus, \leq n}(A^{eff}(\F_\F), \mb{E})$ is both the essential image of $\Mack^{\leq n}(\Fs, \mb{E})$ under \[ \phi^* : \Mack(\Fs, \mb{E}) \to \Fun^{\oplus}(A^{eff}(\F_\F), \mb{E}) \] and the essential image of $\Fun^{\leq n}(A^{eff}(\F), \mb{E})$ under 
	\[ (\kappa^*)^{-1} : \Fun(A^{eff}(\F), \mb{E}) \to \Fun^{\oplus}(A^{eff}(\F_\F), \mb{E}); \] seeing that these two categories coincide is a simple matter of unwrapping the definitions.
	
	Finally, we note that there are explicit formulae for the functors appearing as inverses. The inverse of $R_n$ is left Kan extension along the additive completion functor $A^{eff}(\F) \to A(\F)$ followed by localization into the category of degree $n$ functors. The inverse of ${\Im_n^*}$ is left Kan extension along $\Im_n$. Finally, as explained in the proof of Proposition \ref{prop:connsp}, the inverse of $\iota^*$ is given by precomposition with the connective cover functor $\tau_{\geq 0}$ followed by $n$-excisivization $P_n$.

\end{proof}
	For ease of reference, we include a table matching each of these functor to the location of its definition and the proof that it's an equivalence:
		\begin{center}
			\begin{tabular}{ | l | l | l | }
				\hline
				\bf{Functor}  & \bf{Definition} & \bf{Proof} \\
				\hline
				$\Xi^{\Fs^{\leq n}}$ & \cite[Definition 2.22]{Gla15a} & \cite[Corollary 2.33]{Gla15a} \\
				\hline
				$\phi^*$ & \ref{def:phi} & \ref{thm:phi} \\
				\hline
				$\kappa^*$ & \ref{prop:kappa} & \ref{prop:kappa} \\
				\hline
				$R_n$ & Theorem \ref{thm:anacon} & Theorem \ref{thm:anacon} \\
				\hline
				$\Im_n^*$ & Corollary \ref{cor:resaf} & Proposition \ref{prop:nexcessim} \\
				\hline
				$\iota_*$ & Proposition \ref{prop:connsp} & Proposition \ref{prop:connsp} \\
				\hline

			\end{tabular}
		\end{center}
\begin{rem}
	Let ${\td{F}_\text{surj}}^{\leq n}$ be the full subcategory of $\Fs^{\leq n}$ spanned by the \emph{nonempty} sets. Then clearly 
	\[ \Fs^{\leq n} \simeq {\td{F}_\text{surj}}^{\leq n} \amalg \{\emptyset\}. \]
	This disjoint union extends to a direct sum decomposition of semiadditive categories
	\[ A^{eff}(\F_{\Fs^{\leq n}}) \simeq A^{eff}(\F_{{\td{F}_\text{surj}}^{\leq n}}) \oplus A^{eff}(\F_{\{\emptyset\}}); \]
	in other words, since $A^{eff}(\F_{\{\emptyset\}}) \simeq A^{eff}(\F)$ is the free semiadditive category on one generator, an $\mb{E}$-valued Mackey functor $M$ on $\Fs^{\leq n}$ is the same data as a Mackey functor $\td{M}$ on ${\td{F}_\text{surj}}^{\leq n}$ together with an object $M(\emptyset)$ of $\mb{E}$. This mirrors the observation that an $n$-excisive functor $F : \Sp^{\omega} \to \mb{E}$ decomposes as the direct sum of a reduced $n$-excisive functor $\td{F} : \Sp^{\omega} \to \mb{E}$ - that is, $\td{F}(0) \simeq 0$ - and a constant functor $\mb{const}(F(0))$.
	
	Indeed, if $F$ and $M$ correspond under the equivalence of Theorem \ref{thm:main}, then it's easy to see that $F(\emptyset) \simeq M(\emptyset)$, and thus $\Mack({\td{F}_\text{surj}}^{\leq n})$ - the full subcategory of $\Mack(\Fs^{\leq n})$ of objects $M$ for which $M(\emptyset) \simeq 0$ - is equivalent to the category of \emph{reduced} $n$-excisive functors $\Sp^{\omega} \to \mb{E}$. 
\end{rem}

\bibliographystyle{alpha}
\bibliography{mybib}
\end{document}